\documentclass[12pt]{amsart}
\usepackage{bm}
 \usepackage{graphicx}

 \newtheorem{thm}{Theorem}[section]
 
 \newtheorem{lem}[thm]{Lemma}
 
 \theoremstyle{definition}
 
 \theoremstyle{remark}
 
 \numberwithin{equation}{section}

 \newcommand{\Real}{\mathbb{R}}

 \newcommand{\tr}{\textbf{tr}}

\newcommand{\e}{\epsilon}
\begin{document}

\title[Sharp Harnack inequalities for hypoelliptic diffusions]{Sharp Harnack inequalities for a family of hypoelliptic diffusions}

\author{Paul W.Y. Lee}
\email{wylee@math.cuhk.edu.hk}
\address{Room 216, Lady Shaw Building, The Chinese University of Hong Kong, Shatin, Hong Kong}

\date{\today}

\begin{abstract}
We prove sharp Harnack inequalities for a family of Kolmogorov-Fokker-Planck type hypoelliptic diffusions. 
\end{abstract}

\maketitle

\section{Introduction}

We study a family of hypoelliptic evolution equations which are not uniformly parabolic. The simplest example of this family is given by the following Kolmogorov equation in $\Real^2$: 
\begin{equation}\label{KFKsim}
\rho_t=\rho_{vv}-v\rho_x. 
\end{equation}

The equation (\ref{KFKsim}) was first studied by Kolmogorov in \cite{Ko}. Even though the second order operator $\partial_v^2$ is not hypoelliptic, it was shown in \cite{Ko} that the differential operator defined by the equation (\ref{KFKsim}) is. Motivated by this, H\"ormander showed in \cite{Ho} that second order differential operators $\sum_i X_i^2+Y$ defined by vector fields $\{X_1,...,X_n,Y\}$ is hypoelliptic if these vector fields together with their iterated Lie brackets span the whole space at each point. 

In this paper, we study hypoelliptic evolution equations on $\Real^n\times\Real^n$:
\begin{equation}\label{KFK}
\rho_t=\sum_{i=1}^n(\rho_{v_iv_i}-U_{v_i} \rho_{v_i}-v_i\rho_{x_i}), 
\end{equation}
where $U$ is a smooth function on $\Real^n\times\Real^n$. Hypoelliptic equations, like the ones in (\ref{KFK}), appeared in many differnet branches in applied mathematics. In particular, the family of equations in (\ref{KFK}) contain the linear kinetic Fokker-Planck equations. The long time asymptotic of solutions to the kinetic Fokker-Planck equations is also a very active area of research in the past decades. For a comprehensive introduction, see \cite{Vi} (see also \cite{Ba} for results which are closely related to this paper). For applications of hypoelliptic  equations to mathematical finance, see \cite{CPP, LPP} and reference therein. 

We are interested in differential Harnack inequalities for non-negative solutions to the equation (\ref{KFK}). All previous works known to the author focus on the case $U\equiv 0$. In \cite{PaPo}, scalar differential Harnack inequalities were proved using explicit formulas of the fundamental solutions. In \cite{Ha3}, matrix differential Harnack inequalities were first proved for the equation (\ref{KFK}) with $n=1$ and $U\equiv 0$ (see also \cite{Hu} for an extension to the higher dimensional case with $U\equiv 0$). Matrix differential Harnack inequality was proved in the case of the heat equation by Hamilton in \cite{Ha1} after the scalar version appeared in \cite{LY}. In the case of the Ricci flow, the matrix differential Harnack inequality was proved in \cite{Ha2}. In this case, unlike the case of the heat equation, one can only establish the scalar version using maximum principle by first proving the matrix one. Similar to the case of the Ricci flow, the scalar differential Harnack inequalities we proved for the equations (\ref{KFK}) cannot be established directly using the maximum principle without proving the matrix inequalities. 

One of the key ideas in the proof was coming from the earlier work of the author in \cite{Le1, Le2}. In there, we obtained new differential Harnack inequalities for linear parabolic equations on manifolds with Ricci curvature lower bounds. Another key idea is the use of comparison principle and explicit solutions of matrix Riccati equations. This appeared in the earlier work of the author with his collaborators in \cite{AgLe1, AgLe2, LLZ, Le3}. In that case, the matrix Riccati equations appeared when we considered linearizations of sub-Riemannian geodesic flows. 

Finally, we obtain Harnack inequalities for solutions of (\ref{KFK}) by integrating the differential one along paths satisfying appropriate affine control systems. For this, it is necessary that any two points can be connected by paths satisfying the control system. In our case, this is in fact correct. This non-trivial controllability issue follows from the results in \cite{JuKu} and \cite{JuSu} and this will be discussed in section 4. 

The paper is organized as follows. In section 2, we will discuss the precise statements of the main results. The matrix differential Harnack inequalities are very closely related to matrix Riccati equations. In section 3, we will prove two lemmas using comparison principle and explicit solutions of the matrix Riccati equations. They are needed for the proof of the main results. In section 4, we give the proof of the matrix differential Harnack inequalities for solutions of (\ref{KFK}) and show how they are related to solutions of the matrix Riccati equations. Finally, we give a brief discussion on controllability issues of affine-control systems. We then use it to prove the Harnack inequalities. 

\smallskip

\section*{Acknowledgements}

The work started from the suggestions by Professor Fabrice Baudoin and Professor Wilfrid Gangbo. The problem on searching for appropriate differential Harnack inequalities for hypoelliptic diffusion equations was posted by Professor Baudoin. The reference \cite{PaPo} was also pointed out by him. The author would like to thank them for their advices and encouragement.

\smallskip

\section{The Main Results}

In this section, we introduce notations that will be used throughout this paper and state of our main results. Let $\rho$ be a non-negative solution of the equation 
\begin{equation}\label{hypo}
\dot\rho=\Delta_v\rho-\left<\nabla_vU,\nabla_v \rho\right>-\left<v,\nabla_x\rho\right>. 
\end{equation}
Here $\Delta_x=\sum_{i=1}^n\partial_{x_i}^2$, $\nabla_{x}f=\sum_i (\partial_{x_i}f)\partial_{x_i}$, $\nabla_{v}f=\sum_i (\partial_{v_i}f)\partial_{v_i}$, $\Delta_v=\sum_{i=1}^n\partial_{v_i}^2$, $\left<\cdot,\cdot\right>$ is the Euclidean inner product on $\Real^n$, and $U$ is a smooth function on $\Real^n\times\Real^n$. 

We will assume that the solutions $\rho$ and the function $U$ satisfy the following growth conditions
\begin{equation}\label{growth}
\begin{split}
&|\nabla \rho|\leq c_0e^{-\frac{1}{2}U}(|x|+|v|), \\
&\nabla^2\rho\geq -k_0e^{k_1(|x|^2+|v|^2)}I, \\
&\nabla^2U\leq k_0e^{k_1(|x|^2+|v|^2)}I. 
\end{split}
\end{equation}

Let $h:\Real^n\times\Real^n\to\Real$ be the function defined by 
\begin{equation}\label{h}
h=-\frac{1}{2}\left<v,\nabla_xU\right>+\frac{1}{2}\Delta_v U-\frac{1}{4}|\nabla_v U|^2. 
\end{equation}
We will also assume that the Hessian $\nabla^2 h$ of $h$ satisfies the following lower bound 
\begin{equation}\label{lb}
\nabla^2 h\geq \left(\begin{array}{cc}
-k_1I & O\\
O & -k_2I\end{array}\right),
\end{equation}
where $k_1$ and $k_2$ are non-negative constants. 

The matrix differential Harnack inequalites are of the following form: 
\begin{equation}\label{mDH}
\nabla^2 \log\rho-\frac{1}{2}\nabla^2 U\geq \frac{1}{2}\left(\begin{array}{cc}
-\frac{s_1}{s_0}I & \frac{s_2}{s_0}I\\
\frac{s_2}{s_0}I & -\frac{\dot s_0}{s_0}I
\end{array}\right), 
\end{equation}
where the functions $s_0$, $s_1$, and $s_2$ depend on $k_1$ and $k_2$. 

The following theorem gives the matrix differential Harnack inequalities for the main results. 

\begin{thm}\label{main}
Let $\rho$ be a non-negative solution of the equation (\ref{hypo}). Assume that the solution $\rho$ and the function $U$ satisfy the growth conditions (\ref{growth}). Assume that the function $h$ defined by (\ref{h}) satisfies (\ref{lb}). Then (\ref{mDH}) holds and the functions $s_0$, $s_1$, and $s_2$ are given by the followings: 
\begin{enumerate}
\item if $k_2^2>2k_1>0$, then 
\[
\begin{split}
&s_0(t)=(\lambda_1^2+\lambda_2^2)\sinh(\lambda_1t)\sinh(\lambda_2t)+2\lambda_1\lambda_2(1-\cosh(\lambda_1t)\cosh(\lambda_2t)), \\
&s_1(t)=\lambda_1\lambda_2(\lambda_1^2-\lambda_2^2)(\lambda_1\sinh(\lambda_1t)\cosh(\lambda_2t) -\lambda_2\cosh(\lambda_1t)\sinh(\lambda_2t)), \\
&s_2(t)=\lambda_1\lambda_2((\lambda_1^2+\lambda_2^2)(\cosh(\lambda_1t)\cosh(\lambda_2t)-1)-2\lambda_1\lambda_2\sinh(\lambda_1t)\sinh(\lambda_2t)),\\
& \lambda_1=\sqrt{k_2+\sqrt{k_2^2-2k_1}},\quad \lambda_2=\sqrt{k_2-\sqrt{k_2^2-2k_1}}, 
\end{split}
\]
\item if $k_2^2=2k_1>0$, then 
\[
\begin{split}
&s_0(t)=\sinh^2(\sqrt{k_2}t)-k_2t^2, \\
&s_1(t)=2k_2^{3/2}(\sqrt{k_2}t+\sinh(\sqrt{k_2}t)\cosh(\sqrt{k_2}t)), \\
&s_2(t)=k_2(\sinh^2(\sqrt{k_2}t)+k_2t^2), 
\end{split}
\]
\item if $k_2^2<2k_1$, then 
\[
\begin{split}
&s_0(t)=\mu_2^2\sinh^2\left(\frac{\mu_1}{\sqrt 2}t\right)-\mu_1^2\sin^2\left(\frac{\mu_2}{\sqrt 2}t\right),\\
&s_1(t)=2\sqrt{k_1}\mu_1\mu_2\left(\mu_1\sin\left(\frac{\mu_2}{\sqrt 2}t\right)\cos\left(\frac{\mu_2}{\sqrt 2}t\right)+ \mu_2\cosh\left(\frac{\mu_1}{\sqrt 2}t\right)\sinh\left(\frac{\mu_1}{\sqrt 2}t\right)\right),\\
&s_2(t)=\sqrt{2k_1}\left(\mu_1^2\sin^2\left(\frac{\mu_2}{\sqrt 2}t\right)+\mu_2^2\sinh^2\left(\frac{\mu_1}{\sqrt 2}t\right)\right), \\
& \mu_1=\sqrt{\sqrt{2k_1}+k_2}, \quad \mu_2=\sqrt{\sqrt{2k_1}-k_2}, 
\end{split}
\]
\item if $k_2^2>2k_1=0$, then 
\[
\begin{split}
&s_0(t)=2\sinh^2(\sqrt{2k_2}t)-\sqrt{2 k_2} t\cosh\left(\sqrt{\frac{k_2}{2}}t\right)\sinh\left(\sqrt{\frac{k_2}{2}}t\right), \\
&s_1(t)=-2\sqrt{2k_2^3}\sinh\left(\sqrt{\frac{k_2}{2}}t\right)\cosh\left(\sqrt{\frac{k_2}{2}}t\right), \\
&s_2(t)=-2k_2\sinh^2\left(\sqrt{\frac{k_2}{2}}t\right), 
\end{split}
\]
\item if $k_2^2=2k_1=0$, then 
$s_0(t)=t^4$, $s_1(t)=6t$, and $s_2(t)=3t^2$. 
\end{enumerate}
\end{thm}

It follows from Theorem \ref{main} that 
\begin{equation}\label{scalar}
\frac{d}{dt}f-|\nabla_vf|^2+\left<v,\nabla_x f\right>-h=\Delta_v\log\rho-\frac{1}{2}\Delta_vU\geq -\frac{n\dot s_0}{2s_0},
\end{equation}
where $f=\log\rho-\frac{1}{2}U$. 

Finally, we obtain Harnack inequalities for non-negative solution of (\ref{hypo}) by integrating (\ref{scalar}) along  paths 
$\gamma=(\gamma_x,\gamma_v)$ satisfying the following control system: 
\begin{equation}\label{spcontrol}
\dot\gamma(t)=\sum_i\gamma_{v_i}(t)\partial_{x_i}+\sum_iu_i(t)\partial_{v_i}, 
\end{equation}
where $u:[s,t]\to \Real^n$ is a piecewise constant control. 

Let $c_{s,t}$ be the optimal control cost defined by 
\[
c_{s,t}(x,y)=\inf_{(u,\gamma)}\left[\int_s^t\frac{1}{4}|u(\tau)|^2-h(\gamma(\tau))d\tau\right],
\]
where the infimum is taken over all piecewise constant controls $u$ and the corresponding paths $\gamma$ satisfying $\gamma(s)=x$, $\gamma(t)=y$, and (\ref{spcontrol}). 

Note that for the cost function $c_{s,t}$ to be well-defined, it is necessary that any two points can be connected by paths satisfying the control system (\ref{spcontrol}) (see section 4 for the detail). 

\begin{thm}\label{main2}
Under the same assumptions and notations as in Theorem \ref{main}, the following Harnack inequality holds for any $t>s>0$:
\[
\frac{\rho_t(y)}{\rho_s(x)}\geq \left(\frac{s_0(t)}{s_0(s)}\right)^{-n/2}e^{-c_{s,t}(x,y)+\frac{1}{2}U(y)-\frac{1}{2}U(x)}. 
\]
\end{thm}

\smallskip

\section{Matrix Riccati equations}

In this section, we discuss two lemmas concerning solutions of a matrix Riccati equation which will be used in the proof of Theorem \ref{main}. First, we have 

\begin{lem}\label{RiccatiExist}
Let 
\[
C=\left(\begin{array}{cc}
O & -I\\
O & O
\end{array}\right),\quad D=\left(\begin{array}{cc}
O & O\\
O & 2I
\end{array}\right). 
\]
Assume that $K$ is a non-negative definite. Then there is a solution of the equation 
\begin{equation}\label{Riccati}
\dot N=NC+C^TN+NDN-K
\end{equation}
satisfying $\lim_{t\to 0}N^{-1}=0$, $N\leq0$, and $\dot N\geq 0$. 
\end{lem}

\begin{proof}
If $N$ is a solution the equation (\ref{Riccati}), then $S:=N^{-1}$ is a solution to the equation 
\begin{equation}\label{Riccati2}
\dot S=-CS-SC^T-D+SKS. 
\end{equation}
Therefore, it is enough to show that the solution $S$ of the equation (\ref{Riccati2}) with $S(0)=0$ satisfies $S\leq 0$ and $\dot S\leq 0$. 

By expanding $S$ near $t=0$, we get that 
\[
S=\left(\begin{array}{cc}
-\frac{2t^3}{3}I & -t^2I\\
-t^2I & -2tI+\frac{4t^3}{3}k_2
\end{array}\right)+O(t^4)
\]
as $t\to 0$, where $K=\left(\begin{array}{cc}
K_1 & K_2\\
K_2^T & k_2
\end{array}\right)$. Since the eigenvalues of the matrix $\left(\begin{array}{cc}
-\frac{2t^3}{3}I & -t^2I\\
-t^2I & -2tI+\frac{4t^3}{3}k_2
\end{array}\right)$ are given by 
\[
-\left(\frac{t^3}{3}+t-\frac{2t^3\lambda_i}{3}\right)\pm\sqrt{\left(\frac{t^3}{3}+t-\frac{2t^3\lambda_i}{3}\right)^2-\frac{t^4}{3}+\frac{8t^6}{9}\lambda_i} 
\]
which is negative if $t>0$ is sufficiently small. Here $\lambda_i$ is any eigenvalue of $k_2$. Therefore, $S$ is negative definite for $t>0$ sufficiently small. 

The rest follows from the following comparison theorem of Riccati equation in \cite{Ro}. 

\begin{thm}\label{compare}
Let $S_i$, $i=1, 2$ be solutions of the equations
\[
\dot S_i=A_i+B_iS_i+S_iB_i^T+S_iC_iS_i
\]
with initial conditions $S_1(0)\leq S_2(0)$. Assume that 
\[
\left(\begin{array}{cc}
A_1 & B_1\\
B_1^T & C_1
\end{array}\right)\leq \left(\begin{array}{cc}
A_2 & B_2\\
B_2^T & C_2
\end{array}\right). 
\]
Then $S_1(t)\leq S_2(t)$ for all $t\geq 0$. 
\end{thm}

It follows immediately from Theorem \ref{compare} that $S\leq 0$. By (\ref{Riccati2}), 
\[
\ddot S=-C\dot S-\dot SC^T+\dot SKS+SK\dot S
\]
and $\dot S(0)=-D$. It follows from Theorem \ref{compare} that $\dot S\leq 0$. 
\end{proof}

Next, we write down explicitly the solution of the equation (\ref{Riccati}) in the case when $K=\left(\begin{array}{cc}
k_1I & O\\
O & k_2I
\end{array}\right)$, where $k_1$ and $k_2$ are non-negative. 

\begin{lem}\label{explicit}
Let $k_1$ and $k_2$ be two non-negative number. Let 
\[
S=\frac{1}{2}\left(\begin{array}{cc}
-\frac{s_1}{s_0}I & \frac{s_2}{s_0}I\\
\frac{s_2}{s_0}I & -\frac{\dot s_0}{s_0}I
\end{array}\right), 
\]
where $s_0$, $s_1$, and $s_2$ are defined as in Theorem \ref{main}. 

Then $S$ is a solution of the equation (\ref{Riccati}) with $K=\left(\begin{array}{cc}
k_1I & O\\
O & k_2I
\end{array}\right)$ satisfying $\lim_{t\to 0}S^{-1}(t)=0$. 
\end{lem}

\begin{proof}

The proof relies on the following formula which appeared in \cite{Le}. 

\begin{thm}\label{compute}
Let $M=\left(\begin{array}{cc}
M_1 & M_2\\
M_3 & M_4
\end{array}\right)$ be the fundamental solution of the following system of differential equations
\begin{equation}\label{ode}
\dot M=\left(\begin{array}{cc}
B_1 & B_2\\
-B_3 & -B_1^T 
\end{array}\right)M
\end{equation}
Then $S=(M_1S(0)+M_2)(M_3S(0)+M_4)^{-1}$ is a solution to the matrix Riccati equation $\dot S=B_2+B_1S+SB_1^T+SB_3S$. 
\end{thm}

In our present case, the equation (\ref{ode}) is given by 
\[
\dot M=\left(\begin{array}{cc}
C^T & -K\\
-D & -C 
\end{array}\right)M
\]

First, assume that $k_2^2>2k_1>0$. A computation shows that the eigenvalues of the matrix $\left(\begin{array}{cc}
C^T & -K\\
-D & -C 
\end{array}\right)$ are given by $\lambda_1$, $-\lambda_1$, $\lambda_2$, and $-\lambda_2$. Another computation shows that 
\[
U^{-1}\left(\begin{array}{cc}
C^T & -K\\
-D & -C 
\end{array}\right)U=\left(\begin{array}{cccc}
\lambda_1I & O & O & O\\
O & -\lambda_1I & O & O\\
O & O & \lambda_2I & O\\
O & O & O & -\lambda_2I
\end{array}\right). 
\]
where 
\[
U=\left(\begin{array}{cccc}
-\frac{k_1}{\lambda_1^2}I & -\frac{k_1}{\lambda_1^2}I & -\frac{k_1}{\lambda_2^2}I & -\frac{k_1}{\lambda_2^2}I\\
\frac{k_1-\lambda_1^2k_2}{\lambda_1^3}I & -\frac{k_1-\lambda_1^2k_2}{\lambda_1^3}I & \frac{k_1-\lambda_2^2k_2}{\lambda_2^3}I & -\frac{k_1-\lambda_2^2k_2}{\lambda_2^3}I\\
\frac{1}{\lambda_1}I & -\frac{1}{\lambda_1}I & \frac{1}{\lambda_2}I & -\frac{1}{\lambda_2}I\\
I & I & I & I
\end{array}\right)
\]
and 
\[
U^{-1}=\left(\begin{array}{cccc}
-\frac{k_1}{\lambda_1^2}I & -\frac{k_1}{\lambda_1^2}I & -\frac{k_1}{\lambda_2^2}I & -\frac{k_1}{\lambda_2^2}I\\
\frac{k_1-\lambda_1^2k_2}{\lambda_1^3}I & -\frac{k_1-\lambda_1^2k_2}{\lambda_1^3}I & \frac{k_1-\lambda_2^2k_2}{\lambda_2^3}I & -\frac{k_1-\lambda_2^2k_2}{\lambda_2^3}I\\
\frac{1}{\lambda_1}I & -\frac{1}{\lambda_1}I & \frac{1}{\lambda_2}I & -\frac{1}{\lambda_2}I\\
I & I & I & I
\end{array}\right). 
\]

Therefore, 
\[
M=U\left(\begin{array}{cccc}
e^{\lambda_1t}I & O & O & O\\
O & e^{-\lambda_1t}I & O & O\\
O & O & e^{\lambda_2t}I & O\\
O & O & O & e^{-\lambda_2t}I
\end{array}\right)U^{-1}.
\]
By Theorem \ref{compute} and $\lim_{t\to 0}S(t)^{-1}=0$, $S=M_1M_3^{-1}$. Another computation gives the result in this case. 

In the case $k_2^2=2k_1>0$, the eigenvalues are given by $\pm\sqrt{k_2}$. A computation shows that 
\[
U^{-1}\left(\begin{array}{cc}
C^T & -K\\
-D & -C 
\end{array}\right)U=\left(\begin{array}{cccc}
\sqrt{k_2}I & I & O & O\\
O & \sqrt{k_2}I & O & O\\
O & O & -\sqrt{k_2}I & I\\
O & O & O & -\sqrt{k_2}I
\end{array}\right). 
\]
where 
\[
U=\left(\begin{array}{cccc}
\frac{k_2^2}{2}I & -k_2^{3/2}I & \frac{k_2^2}{2}I & k_2^{3/2}I\\
\frac{k_2^{3/2}}{2}I & \frac{k_2}{2}I & -\frac{k_2^{3/2}}{2}I & \frac{k_2}{2}I\\
-\sqrt{k_2}I & I & \sqrt{k_2}I & I\\
-k_2I & O & -k_2I & O
\end{array}\right)
\]
and 
\[
U^{-1}=\left(\begin{array}{cccc}
O & \frac{1}{2k_2^{3/2}}I & -\frac{1}{4k_2^{1/2}}I & -\frac{1}{2k_2}I\\
-\frac{1}{2k_2^{3/2}}I & \frac{1}{2k_2}I & \frac{1}{4}I & -\frac{1}{4k_2^{1/2}}I\\
O & -\frac{1}{2k_2^{3/2}}I & \frac{1}{4k_2^{1/2}}I & -\frac{1}{2k_2}I\\
\frac{1}{2k_2^{3/2}}I & \frac{1}{2k_2}I & \frac{1}{4}I & \frac{1}{4k_2^{1/2}}I
\end{array}\right). 
\]
Therefore, 
\[
M=U\left(\begin{array}{cccc}
e^{\sqrt{k_2}t}I & te^{\sqrt{k_2}t}I & O & O\\
O & e^{\sqrt{k_2}t}I & O & O\\
O & O & e^{-\sqrt{k_2}t}I & te^{-\sqrt{k_2}t}I\\
O & O & O & e^{-\sqrt{k_2}t}I
\end{array}\right)U^{-1}.
\]
Therefore, by using $S=M_1M_3^{-1}$, the result in this case follows. 

Next, assume that $k_2^2<2k_1$. The characteristic polynomial of the matrix $\left(\begin{array}{cc}
C^T & -K\\
-D & -C 
\end{array}\right)$ is given by $(x^2+\sqrt 2 \mu_1 x+\sqrt{2k_1})(x^2-\sqrt 2\mu_1 x+\sqrt{2k_1})$, where $\mu_1=\sqrt{\sqrt{2k_1}+k_2}$. A computation shows that 
\[
U^{-1}\left(\begin{array}{cc}
C^T & -K\\
-D & -C 
\end{array}\right)U=\left(\begin{array}{cccc}
-\sqrt 2 \mu_1 I & I & O & O\\
-\sqrt{2k_1}I & O & O & O\\
O & O & \sqrt 2\mu_1 I & I\\
O & O & -\sqrt{2k_1}I & O
\end{array}\right). 
\]
where 
\[
U=\left(\begin{array}{cccc}
-k_1I & \sqrt{k_1}\mu_1 I & -k_1I & -\sqrt{k_1}\mu_1 I\\
-\sqrt{k_1}\mu_1 I & \sqrt{\frac{k_2}{2}}I & \sqrt{k_1}\mu_1 I & \sqrt{\frac{k_2}{2}}I\\
O & I & O & I\\
-\sqrt{2k_1}I & O & -\sqrt{2k_1}I & O
\end{array}\right)
\]
and 
\[
U^{-1}=\left(\begin{array}{cccc}
O & -\frac{1}{2\sqrt{k_1}\mu_1}I & \frac{1}{2\sqrt{2}\mu_1} & -\frac{1}{2\sqrt{2k_1}}I\\
\frac{1}{2\sqrt{k_1}\mu}I & O & \frac{1}{2}I & -\frac{1}{2\sqrt{2}\mu_1}I\\
O & \frac{1}{2\sqrt{k_1}\mu_1}I & -\frac{1}{2\sqrt{2}\mu_1}I & -\frac{1}{2\sqrt{2k_1}}I\\
-\frac{1}{2\sqrt{k_1}\mu_1}I & O & \frac{1}{2}I & \frac{1}{2\sqrt{2}\mu_1}I
\end{array}\right). 
\]

Therefore, 
\[
M=U\left(\begin{array}{cc}
A_1 & O\\
O & A_2
\end{array}\right)U^{-1},
\]
where $\mu_2=\sqrt{\sqrt{2k_1}-k_2}$, 
\[
A_1=e^{-\frac{\mu_1t}{\sqrt{2}}}\left(\begin{array}{cc}
\left(\cos\left(\frac{\mu_2}{\sqrt 2}t\right) -\frac{\mu_1}{\mu_2}\sin\left(\frac{\mu_2}{\sqrt 2}t\right) \right)I & \frac{\sqrt 2}{\mu_2}\sin\left(\frac{\mu_2}{\sqrt 2}t\right)I\\
-\frac{ 2\sqrt{k_1}}{\mu_2}\sin\left(\frac{\mu_2}{\sqrt 2}t\right)I & \left(\cos\left(\frac{\mu_2}{\sqrt 2}t\right) +\frac{\mu_1}{\mu_2}\sin\left(\frac{\mu_2}{\sqrt 2}t\right) \right)I\\
\end{array}\right),
\]
and 
\[
A_2=e^{\frac{\mu_1t}{\sqrt{2}}}\left(\begin{array}{cc}
\left(\cos\left(\frac{\mu_2}{\sqrt 2}t\right) +\frac{\mu_1}{\mu_2}\sin\left(\frac{\mu_2}{\sqrt 2}t\right) \right)I & \frac{\sqrt 2}{\mu_2}\sin\left(\frac{\mu_2}{\sqrt 2}t\right)I\\
-\frac{ 2\sqrt{k_1}}{\mu_2}\sin\left(\frac{\mu_2}{\sqrt 2}t\right)I & \left(\cos\left(\frac{\mu_2}{\sqrt 2}t\right) -\frac{\mu_1}{\mu_2}\sin\left(\frac{\mu_2}{\sqrt 2}t\right) \right)I\\
\end{array}\right). 
\]

If we assume that $k_2^2>2k_1=0$, then computation shows that the eigenvalues of the matrix $\left(\begin{array}{cc}
C^T & -K\\
-D & -C 
\end{array}\right)$ are given by $0$ and $\pm\sqrt {2k_2}$. Another computation shows that 
\[
U^{-1}\left(\begin{array}{cc}
C^T & -K\\
-D & -C 
\end{array}\right)U=\left(\begin{array}{cccc}
O & I & O & O\\
O & O & O & O\\
O & O & \sqrt{2k_2}I & O\\
O & O & O & -\sqrt{2k_2}I
\end{array}\right). 
\]
where 
\[
U=\left(\begin{array}{cccc}
O & -k_2I & O & O\\
O & O & -\sqrt{\frac{k_2}{2}}I & \sqrt{\frac{k_2}{2}}I\\
I & O & \frac{1}{\sqrt{2k_2}}I & -\frac{1}{\sqrt{2k_2}}I\\
O & I & I & I
\end{array}\right)
\]
and 
\[
U^{-1}=\left(\begin{array}{cccc}
O & \frac{1}{k_2}I & I & O\\
-\frac{1}{k_2}I & O & O & O\\
\frac{1}{2k_2}I & -\frac{1}{\sqrt{2k_2}}I & O & \frac{1}{2}I\\
\frac{1}{2k_2}I & \frac{1}{\sqrt{2k_2}}I & O & \frac{1}{2}I\\
\end{array}\right). 
\]

Therefore, 
\[
M=U\left(\begin{array}{cccc}
I & tI & O & O\\
O & I & O & O\\
O & O & e^{\sqrt{2k_2}t}I & O\\
O & O & O & e^{-\sqrt{2k_2}t}I
\end{array}\right)U^{-1}.
\]

If we assume that $k_2=2k_1=0$, then computation shows that 0 is the only eigenvalue of the matrix $\left(\begin{array}{cc}
C^T & -K\\
-D & -C 
\end{array}\right)$. Another computation shows that 
\[
M=\left(\begin{array}{cccc}
I & O & O & O\\
-tI & I & O & O\\
\frac{1}{3}t^3I & -t^2I & I & tI\\
t^2I & -2tI & O & I
\end{array}\right).
\]
\end{proof}

\smallskip

\section{The matrix differential Harnack inequalities}

In this section, we give the proof of Theorem \ref{main}. It follows immediately from Lemma \ref{explicit} and the following theorem. 

\begin{thm}\label{matrixmain}
Let $\rho$ be a non-negative solution of the equation (\ref{hypo}) satisfying the growth conditions (\ref{growth}). Assume that the function $h$ defined by (\ref{h}) satisfies $\nabla^2 h\geq -K$ for some non-negative definite matrix $K$. Then 
\[
\nabla^2 \log\rho-\frac{1}{2}\nabla^2 U\geq N
\]
where $N$ is a solution of the matrix Riccati equation 
\[
\dot N=NC+C^TN+NDN-K
\]
satisfying the condition $\lim_{t\to 0}N^{-1}=0$. 
\end{thm}

\begin{proof}
By replacing $\rho$ by $\rho+\delta$, we can assume that $\rho$ is bounded below by a positive constant. 
Let $f=\log\rho$. Then a computation shows that 
\begin{equation}\label{f}
\dot f=\Delta_v f+\left|\nabla_vf-\frac{1}{2}\nabla_vU\right|^2-\left<v,\nabla_xf\right>-\frac{1}{4}|\nabla_v U|^2 
\end{equation}

Let $g=f-\frac{1}{2}U$. Then 
\begin{equation}\label{g}
\dot g=\Delta_v g+\left|\nabla_vg\right|^2-\left<v,\nabla_xg\right> +h. 
\end{equation}

By differentiating the above equation, we obtain 
\[
\begin{split}
&\dot g_{x_ix_j}=\Delta_v g_{x_ix_j}+2\left<\nabla_v g_{x_j},\nabla_vg_{x_i}\right>\\
&+2\left<\nabla_v g,\nabla_vg_{x_ix_j}\right>-\left<v,\nabla_xg_{x_ix_j}\right> +h_{x_ix_j}, 
\end{split}
\]
\[
\begin{split}
&\dot g_{x_iv_j}=\Delta_v g_{x_iv_j}+2\left<\nabla_v g_{v_j},\nabla_vg_{x_i}\right>\\
&+2\left<\nabla_v g,\nabla_vg_{x_iv_j}\right>-\left<v,\nabla_xg_{x_iv_j}\right> -g_{x_ix_j} +h_{x_iv_j}, 
\end{split}
\]
and 
\[
\begin{split}
&\dot g_{v_iv_j}=\Delta_v g_{v_iv_j}+2\left<\nabla_v g_{v_j},\nabla_vg_{v_i}\right>\\
&+2\left<\nabla_v g,\nabla_vg_{v_iv_j}\right>-g_{x_iv_j}-g_{v_ix_j}-\left<v,\nabla_x g_{v_iv_j}\right> +h_{v_iv_j}. 
\end{split}
\]

Let $M$ be the matrix defined by $M=\left(\begin{array}{cc}
g_{xx} & g_{xv}\\
g_{xv}^T & g_{vv}
\end{array}\right)$, where $g_{xx}$, $g_{xv}$, and $g_{vv}$ are $n\times n$ matrices with $ij$-th entry equal to $g_{x_ix_j}$, $g_{x_iv_j}$, and $g_{v_iv_j}$, respectively. It follows from the above equations that  
\[
\dot M=\Delta_v M+\left<2\nabla_vg,\nabla_v M\right>-\left<v,\nabla_x M\right>+MC+C^TM+MDM+\nabla^2 h
\]
where $C=\left(\begin{array}{cc}
O & -I\\
O & O
\end{array}\right)$ and $D=\left(\begin{array}{cc}
O & O\\
O & 2I
\end{array}\right)$. 

Let $N$ be the solution of the following matrix Riccati equation given by Lemma \ref{RiccatiExist}
\[
\dot N=NC+C^TN+NDN-K
\]
which satisfies the condition $\lim_{t\to 0}N^{-1}=0$. 

Let $M_\e=M-N+\e \phi I$, where $\phi$ is a positive function on $\Real^n\times\Real^n$ to be chosen. Then, by using $\nabla^2 h\geq -K$, we obtain 
\[
\begin{split}
&\dot M_\e\geq \Delta_v M_\e-\e\Delta_v\phi \, I+\left<2\nabla_vg,\nabla_v M_\e\right>-\e\left<2\nabla_vg,\nabla_v \phi\right>I\\
&-\left<v,\nabla_x M_\e\right>+\e\left<v,\nabla_x \phi\right>I+(M_\e-\e\phi I)C+C^T(M_\e-\e\phi I)\\
&+M_\e DM_\e+M_\e DN-\e\phi M_\e D+NDM_\e-\e\phi ND\\
&-\e\phi DM_\e-\e\phi DN+\e^2\phi^2 D+\e\dot\phi I. 
\end{split}
\]

By assumptions, we have $|\nabla \rho|\leq c_0e^{-\frac{1}{2}U}(|x|+|v|)$, $\nabla^2\rho\geq -k_0e^{k_1(|x|^2+|v|^2)}I$, and $\nabla^2U\leq k_0e^{k_1(|x|^2+|v|^2)}I$. Therefore, 
\[
\begin{split}
&\nabla^2g=\frac{1}{\rho}\nabla^2\rho-\frac{1}{\rho^2}\nabla\rho\otimes\nabla\rho -\frac{1}{2}\nabla^2 U\\
&\geq -\frac{k_0}{\delta}e^{k_1(|x|^2+|v|^2)}I-\frac{1}{\delta^2}|\nabla\rho|^2I -\frac{1}{2}\nabla^2 U\\
&\geq -\left(\frac{1}{\delta}+\frac{1}{2}\right)k_0 e^{k_1(|x|^2+|v|^2)}I-\frac{c_0^2e^{-U}}{\delta^2}(|x|+|v|)^2I. 
\end{split}
\]

Since $\phi$ will be chosen such that $\phi=k_3(t)e^{k_4(t)(|x|^2+|v|^2)}$ with $k_4(t)> k_1$, the function $(x,v,V)\mapsto \left<M_\e(t, x,v)V,V\right>$ achieves its infimum at an interior point $(x,v,V)$ in $\Real^n\times\Real^n\times S^{2n}$ for each fixed $t$. Moreover, $M_\e\geq 0$ for all small enough $t\leq t'$ since $N\to -\infty$ as $t\to 0$. Note that $t'$ depends only on the bound of $M$ only. 

At the first time $t_0$ when $\left<M_\e(x,v) V,V\right>=0$ for some point $(x_0,v_0, V_0)$, we have 
\[
\begin{split}
&0\geq \left<\dot M_\e(x_0,v_0) V_0,V_0\right>\geq -\e\Delta_v\phi -\e\left<2\nabla_vg,\nabla_v \phi\right>\\
&+\e\left<v,\nabla_x \phi\right> -2\left<CV,V\right> -2 \left<NDV,V\right>+\e\dot\phi . 
\end{split}
\]

Let $\varphi=\log\phi$. Then the above inequality becomes 
\[
\begin{split}
&0\geq \dot\varphi-\Delta_v\varphi-|\nabla_v\varphi|^2-\left<2\nabla_vg,\nabla_v \varphi\right>\\
&+\left<v,\nabla_x \varphi\right>-2 \left<CV,V\right> -2 \left<NDV,V\right>. 
\end{split}
\]

Since $\dot N\geq 0$, we can assume that $\left<CV,V\right>+\left<NDV,V\right>\leq c_0$ for $t>t'$ for some $c_0>0$. Therefore, we obtain a contradiction if $\varphi$ satisfies
\[
\begin{split}
&\dot\varphi-\Delta_v\varphi-|\nabla_v\varphi|^2-\left<2\nabla_vg,\nabla_v \varphi\right>+\left<v,\nabla_x \varphi\right>-2c_0>0 
\end{split}
\]

Assume that $\varphi=\frac{a(t)}{2}|x|^2+\frac{b(t)}{2}|v|^2+c(t)$ with $a, b, c\geq 0$. A computation shows that  
\[
\begin{split}
&\dot\varphi-\Delta_v\varphi-|\nabla_v\varphi|^2-\left<2\nabla_vg,\nabla_v \varphi\right>+\left<v,\nabla_x \varphi\right>\\&=\frac{\dot a(t)}{2}|x|^2+\frac{\dot b(t)}{2}|v|^2+\dot c(t)-n b(t)\\
&-b(t)^2|v|^2-2b(t)\left<\nabla_vg,v\right>+a(t)\left<x,v\right>.
\end{split}
\]

By assumption, $|\nabla \rho|\leq c_0e^{-\frac{1}{2}U}(|x|+|v|)\leq \frac{c_0}{\delta}\rho e^{-\frac{1}{2}U}(|x|+|v|)$. Therefore, $|\nabla_v g|$ satisfies $|\nabla_v g|\leq c_1(|x|+|v|)$ and we have 
\[
\begin{split}
&\dot\varphi-\Delta_v\varphi-|\nabla_v\varphi|^2-\left<2\nabla_vg,\nabla_v \varphi\right>+\left<v,\nabla_x \varphi\right>\\&\geq \frac{\dot a(t)}{2}|x|^2+\left(\frac{\dot b(t)}{2}-b(t)^2-2c_1b(t)\right)|v|^2\\
&+\dot c(t)-n b(t)-(2b(t)c_1+a(t))|x||v|\\
&\geq \left(\frac{\dot b(t)}{2}-(1+c_1^2c_3+c_4)b(t)^2-\frac{c_1^2}{c_4}-\frac{1}{2c_2}\right)|v|^2\\
&+\left(\frac{\dot a(t)}{2}-\frac{c_2a(t)^2}{2}-\frac{1}{c_3}\right)|x|^2+\dot c(t)-n b(t). 
\end{split}
\]

Let $a(t)=\frac{1}{2c_2(T-t)}$ and $c_3=16c_2T^2$. Then 
\[
\begin{split}
\frac{\dot a(t)}{2}-\frac{c_2a(t)^2}{2}-\frac{1}{c_3}&=\frac{1}{8c_2(T-t)^2}-\frac{1}{16c_2T^2}\geq 0. 
\end{split}
\]

Let $b(t)=\frac{1}{4(1+c_1^2c_3+c_4)(T-t)}$, $c_4=8c_1^2T^2$, and $c_2=4T^2$. Then 
\[
\begin{split}
&\frac{\dot b(t)}{2}-(1+c_1^2c_3+c_4)b(t)^2-\frac{c_1^2}{c_4}-\frac{1}{2c_2}\\
&\geq\frac{1}{4T^2}-\frac{c_1^2}{c_4}-\frac{1}{2c_2}\geq 0. 
\end{split}
\]

By letting $c(t)=-\frac{n}{4(1+c_1^2c_3+c_4)}\log(T-t)+(2c_0+\delta)t$. Then 
\[
\dot\varphi-\Delta_v\varphi-|\nabla_v\varphi|^2-\left<2\nabla_vg,\nabla_v \varphi\right>+\left<v,\nabla_x \varphi\right>>2c_0.  
\]
as claimed. Finally, the growth condition on the function $\phi$ mentioned earlier can be achieved by choosing $T$ small enough. 

\end{proof}

\smallskip

\section{Optimal control problems and the Harnack inequalities}

In this section, we give the proof of Theorem \ref{main2}. For this, we first recall the results on the controllability of the following affine control system on $\Real^n$:
\begin{equation}\label{control}
\dot x(t)=X_0(x(t))+u_1(t)X_1(x(t))+...+u_m(t)X_m(x(t)). 
\end{equation}

A control $u:[0,T]\to\Real^k$ if there exists a partition of $[0,T]$ such that $u$ is constant on each interval of the partition. A point $y$ is reachable from $x$ at time $T$ by the control system (\ref{control}) if there is a piecewise constant control $u$ such that the corresponding solution of (\ref{control}) satisfies $x(0)=x$ and $x(T)=y$. Let $\mathcal A(x,T)$ be the set of points which are reachable from $x$ at time $T$ by the control system (\ref{control}). Let $\mathcal A(x,\leq T)=\cup_{0\leq t\leq T}\mathcal A(x, t)$. The control system (\ref{control}) is strongly controllable if for any time $T>0$ and any point $x$, $\mathcal A(x,\leq T)=M$. It is exact time controllable if $\mathcal A(x, T)=M$ for each $x$ in $M$ and each $T>0$. The following result is taken from \cite{JuKu} (see also \cite{Ju}). 

\begin{thm}\label{strong}
Assume that the vector field $X_0$ is of the form $X_0=p_0+p_1+...+p_k$, where $k$ is odd and each component of $p_i$ is a homogeneous polynomial of degree $i$. Assume also that $X_1, ..., X_m$ are constant vector fields. Suppose that the smallest vector space containing $X_1,...,X_m$ and invariant under $p_k$ is $\Real^n$. Then 
the control system (\ref{control}) is strongly controllable. 
\end{thm}

For affine control systems (\ref{control}), exact time controllability follows from strong controllability and a result in \cite{JuSu} (see also \cite[Chapter 3]{Ju}). 

Finally, we give the proof of Theorem \ref{main2}. 

\begin{proof}[Proof of Theorem \ref{main2}]
Let $\gamma$ be a path in $\Real^n\times\Real^n$. By (\ref{g}) and Theorem \ref{matrixmain}, 
\[
\frac{d}{dt} g(\gamma(t))\geq \tr(N_3)+\left|\nabla_vg\right|^2_{\gamma(t)}-\left<v,\nabla_xg\right>_{\gamma(t)} +dg(\dot\gamma(t))+h(\gamma(t)). 
\]

Let $V=(V_x,V_v)$ be a vector in $\Real^n\times\Real^n$. The infimum $\inf_{p}(|p_v|^2-\left<v, p_x\right>+\left<p_x,V_x\right>+\left<p_v,V_v\right>)$ is $-\infty$ unless $V_x=v$. Under this assumption, the above infimum becomes $-\frac{1}{4}|V_v|^2$. Therefore, if we assume that $\gamma$ is a path defined by the control system (\ref{control}) with a piecewise constant control $u$, $X_0(x,v)=v_1\partial_{x_1}+...+v_n\partial_{x_n}$, and $X_i(x,v)=\partial_{v_i}$, then 
\[
\frac{d}{dt} g(\gamma(t))\geq \tr(N_3)-\frac{1}{4}|u(t)|^2+h(\gamma(t)). 
\]

Assume that $\gamma(s)=x$ and $\gamma(t)=y$. Note that such a path exists since the control system (\ref{control}), in this case, is exact time controllable by Theorem \ref{strong} and the remark after the statement of Theorem \ref{strong}. It follows that 
\[
g_t(y)-g_s(x)\geq \int_s^t\tr(N_3)d\tau-\int_s^t\frac{1}{4}|u(\tau)|^2-h(\gamma(\tau))d\tau. 
\]

By taking the infimum over all these paths $\gamma$, the result follows. 
\end{proof}

\end{document}